\documentclass[twoside,a4paper]{article}

\usepackage{theorem}
\usepackage{amssymb}
\usepackage{graphicx}
\usepackage{amsmath}
\usepackage{subfigure}
\usepackage{enumerate}

\newtheorem{theorem}{\sc Theorem}[section]

\newtheorem{definition}{\sc Definition}[section]
\newtheorem{example}{\sc Example}[section]

\newtheorem{observation}{\sc Observation}[section]

\newcommand{\qed}{\hskip 10pt $\Box$}

\newenvironment{proof}{\par {\sc Proof.\hskip 5pt}}{\hfill \qed \par}

\newcommand{\R}{{\mathbb R}}
\newcommand{\N}{{\mathbb N}}

\newcommand{\Z}{{\mathbb Z}}
\newcommand{\conv}{{\rm conv}}

\newcommand{\intt}{\mathrm{int}}

\newcommand{\spann}{\mathrm{span}}

\title{Maximal integral simplices with no interior integer points}

\author{Kent Andersen
    \and Christian Wagner
    \and Robert Weismantel
}

\begin{document}

\maketitle

\begin{abstract}
In this paper, we consider integral maximal lattice-free simplices.
Such simplices have integer vertices and contain integer points in the
relative interior of each of their facets, but no integer point is 
allowed in the full interior.
In dimension three, we show that any integral maximal lattice-free
simplex is equivalent to one of seven simplices up to unimodular
transformation.
For higher dimensions, we demonstrate that the set of integral maximal
lattice-free simplices with vertices lying on the coordinate axes is
finite. This gives rise to a conjecture that the total number of
integral maximal lattice-free simplices is finite for any dimension.

\end{abstract}


\section{Introduction}

In dimension $d \in \N$, a simplex is defined to be the convex hull of
$d+1$ affinely independent points. It is called integral if its
vertices have integer coordinates. Integral simplices have been
studied in several contexts. In particular, the literature is rich in
investigations of integral simplices that contain no other lattice
points besides the vertices - neither on the boundary nor in the
interior (see e.g. Reznick \cite{reznick}, Scarf \cite{scarf}, Seb\"{o}
\cite{sebo}).
This notion of lattice-freeness is an interesting concept which has
proven to be valuable for some problems in integer programming and
combinatorics. Most notably, this notion is used for primal integer
programming and the study of neighbors of the origin (see Scarf
\cite{scarf}).

In this paper, we consider a different notion of
lattice-freeness. The application we have in mind is to use integral
simplices as a tool to generate cutting planes for (mixed) integer
linear problems (see e.g. \cite{anlowe2, anlowewo, balasint, coma,
  dewo,zambelli}). For this purpose, we employ the notion of 
lattice-freeness introduced by Lov\'{a}sz \cite{lovasz}.

\begin{definition}
A simplex $S \subseteq \R^d$ is called lattice-free if $\intt(S)
\cap \Z^d = \emptyset$.
\end{definition}

To obtain deep cutting planes, we look for integral lattice-free
simplices which are maximal with respect to inclusion and call them
\emph{integral maximal lattice-free simplices}. A well-known result
of Lov\'{a}sz \cite{lovasz} is that for such simplices each facet
contains an integer point in its relative interior.

The application to cutting plane generation requires however to have
an explicit list of integral maximal lattice-free simplices available.
The partial knowledge about structural properties of such bodies is
definitely not enough.

In dimension two, it can easily be verified that any integral maximal
lattice-free simplex can be unimodularly transformed to
$\conv\left((0,0)^T,(2,0)^T,(0,2)^T\right)$. However, to the best of
our knowledge, a characterization of integral maximal lattice-free
simplices in higher dimensions is not known. Moreover, it is not known
if their number is finite.

We note that a recent paper of
Treutlein \cite{treutlein} shows finiteness of integral maximal
lattice-free simplices in dimension three.
In this paper, we extend on this result: we completely characterize
integral maximal
lattice-free simplices in dimension three and show that - up to
unimodular transformation - only seven different simplices exist.
Furthermore, for a special class of integral maximal lattice-free
simplices, namely simplices with vertices on the coordinate axes,
we argue that their number is finite for any dimension
$d \in \N$.
This gives rise to the conjecture that the total number of integral
maximal lattice-free simplices is finite, in general.

Section \ref{3d} is dedicated to the analysis of the three
dimensional case. Extensions are considered in Section
\ref{ext}.
Sections \ref{case.dist.>=2} and \ref{case.dist.=1} contain details of
our proof technique.



\section{Simplices in dimension three} \label{3d}

Let $S \subseteq \R^3$ be an integral maximal lattice-free simplex.
We assume that $S=\conv(0,v^1,v^2,v^3)$, where $v^1,v^2$, and $v^3$
are integer vectors. By computing the Hermite normal form of the
matrix $(v^1,v^2,v^3)$, it can be assumed that $v^1=(a,0,0)^T$,
$v^2=(b,c,0)^T$, and $v^3=(d,e,f)^T$ such that all coefficients
$a,b,c,d,e$, and $f$ are integer and, in addition, $a,c,f > 0$,
$0 \leq b < c$, $0 \leq d < f$, and $0 \leq e < f$ hold (see
Schrijver \cite{Schrijver}).
Furthermore, we can assume that $c \not = 1$ and $f \not = 1$ since
for $c = 1$ it follows $b = 0$ and thus the facet spanned by the three
points $0,v^1,v^2$ does not contain an interior integer point and
therefore $S$ is not maximal lattice-free. On the other hand, for
$f = 1$, $S$ is contained in the split $\{x \in \R^3 : 0 \leq x_3 \leq
1\}$ which is a contradiction to its maximality. Hence, we 
have $c,f \geq 2$. In the remainder of this paper we work with the
following inequality representation of $S$:
\begin{equation} \label{simp.descr.}
  \begin{pmatrix} 0 & 0 & -1 \\ 0 & -f & e \\ -cf & bf & cd - be \\
    cf & f(a-b) & e(b-a) + c(a-d) \end{pmatrix}
  \begin{pmatrix} x_1 \\ x_2 \\ x_3 \end{pmatrix} \leq
  \begin{pmatrix} 0 \\ 0 \\ 0 \\ acf \end{pmatrix}.
\end{equation}
Our proof strategy is based on partitioning the set of potential
integral maximal lattice-free
simplices according to relations among the unknowns $a,b,c,d,e$
and $f$. For each of the subcases we then manage to compute upper
bounds on $a,c$, and $f$. Once this has been established, the integral
maximal lattice-free simplices can be computed by enumeration.
The enumeration provides a list of simplices which must then be checked
for unimodular equivalence.


\begin{definition} \label{unimod.trans}
Two sets $S,T \in \R^d$ are unimodularly equivalent if there exist a
unimodular matrix $M \in \Z^{d \times d}$ and a vector $v
\in \Z^d$ such that $T = MS + v$, where $MS := \{Ms \in \R^d : s
\in S\}$.
\end{definition}

For integral maximal lattice-free simplices $S=\conv(s^1,s^2,s^3,s^4)$
and $T=\conv(t^1,t^2,t^3,t^4)$ in $\R^3$ it follows that they are
unimodularly equivalent if there exist a matrix $M \in
\Z^{3 \times 3}$ with $|\det(M)| = 1$ and a vector $v \in \Z^3$ such
that $s^j = Mt^{\sigma(j)} + v$ for all $j = 1,2,3,4$, where $\sigma(j)$
is a permutation.
\bigskip

We distinguish our analysis into the two major cases $a \geq 2$ and
$a = 1$. For $a \geq 2$ there are two integer points which play a key
role in our subcase analysis. In the following let
\begin{equation*}
k := \left\lceil \frac{ac+b-a}{c} \right\rceil -1.
\end{equation*}
The distinctions in our subcase analysis for $a \geq 2$ are based on
the locations of the points $(1,1,1)$ and $(k,1,1)$ relative to $S$.
Geometrically, this can be interpreted as follows:
Firstly, we investigate simplices where the point $(1,1,1)$ either lies
on or violates the fourth facet of \eqref{simp.descr.}. Afterwards,
we consider the opposite case and divide, secondly, into simplices
where the point $(k,1,1)$ either lies on or violates the third facet
of \eqref{simp.descr.} and simplices where this is not the case.

Here is the structure of the case distinction for $a \geq 2$ with the
corresponding bounds on $a,c$, and $f$:

\begin{enumerate}[I)]
  \item $cf + f(a-b) + e(b-a) + c(a-d) \geq acf$ \\
    \hspace*{0.5cm} \big( means: $(1,1,1)$ either lies on or violates
    the fourth facet of \eqref{simp.descr.} \big)
  \begin{enumerate}[1)]
    \item $b \geq a$ $\qquad \qquad \Rightarrow$ no integral maximal
      lattice-free simplex possible
    \item $b < a$
    \begin{enumerate}[i)]
      \item $(a,c) = (2,2)$ $\qquad \qquad \Rightarrow (a,c,f) \leq
        (2,2,4)$
      \item $(a,c) \not = (2,2)$ $\qquad \qquad \Rightarrow (a,c,f)
        \leq (6,6,6)$
    \end{enumerate}
  \end{enumerate}
  \item $cf + f(a-b) + e(b-a) + c(a-d) < acf$ \\
    \hspace*{0.5cm} \big( means: $(1,1,1)$ strictly satisfies the
    fourth facet of \eqref{simp.descr.} \big)
    \begin{enumerate}[1)]
      \item $-cfk + bf + cd - be \geq 0$ \\
        \hspace*{0.5cm} \big( means: $(k,1,1)$ either lies on or
        violates the third facet of \eqref{simp.descr.} \big)
      \begin{enumerate}[i)]
        \item $(a,c) = (2,2)$ $\qquad \qquad \Rightarrow (a,c,f)
          \leq (2,2,8)$
        \item $(a,c) \not = (2,2)$ $\qquad \qquad \Rightarrow
          (a,c,f) \leq (3,18,6)$
      \end{enumerate}
      \item $-cfk + bf + cd - be < 0$ \\
        \hspace*{0.5cm} \big( means: $(k,1,1)$ strictly satisfies
        the third facet of \eqref{simp.descr.} \big)
      \begin{enumerate}[i)]
        \item $e > 0$ $\qquad \qquad \Rightarrow$ no integral
          maximal lattice-free simplex possible
        \item $e = 0$
        \begin{enumerate}[A)]
          \item $a = b$ $\qquad \qquad \Rightarrow (a,c,f) \leq
            (3,6,2)$
          \item $a < b$ $\qquad \qquad \Rightarrow (a,c,f) \leq
            (2,12,6)$
          \item $a > b$ $\qquad \qquad \Rightarrow (a,c,f) \leq
            (6,6,6)$
        \end{enumerate}
      \end{enumerate}
    \end{enumerate}
\end{enumerate}

The analysis of the subcases is technical and tedious, but not
complicated, in principle. The complete analysis is given in Section
\ref{case.dist.>=2}.
In summary, the analysis shows that any integral maximal
lattice-free simplex in $\R^3$ with $a \geq 2$ satisfies $a \leq 6$,
$c \leq 18$, and $f \leq 8$.
\bigskip

The case where $a = 1$ must be treated differently. Here, the integer
point $(1,1,1)$ and the unknown parameter $e$ play a key role. The
structure of the case distinction for $a = 1$ with the corresponding
bounds on $c$ and $f$ is shown below.

\begin{enumerate}[I)]
  \item $cf + f(1-b) + e(b-1) + c(1-d) \geq cf$ \\
    \hspace*{0.5cm} \big( means: $(1,1,1)$ either lies on or violates
    the fourth facet of \eqref{simp.descr.} \big) \\
    $\Rightarrow$ no integral maximal lattice-free simplex possible
  \item $cf + f(1-b) + e(b-1) + c(1-d) < cf$ \\
    \hspace*{0.5cm} \big( means: $(1,1,1)$ strictly satisfies the
    fourth facet of \eqref{simp.descr.} \big)
    \begin{enumerate}[1)]
      \item $e = 0$ $\qquad \qquad \Rightarrow (c,f) \leq (8,16)$
      \item $e > 0$
      \begin{enumerate}[i)]
        \item $c \leq e$ $\qquad \qquad \Rightarrow (c,f) \leq (6,12)$ 
        \item $c > e$ $\qquad \qquad \Rightarrow$ reducible to
          the case $a \geq 2$
      \end{enumerate}
    \end{enumerate}
\end{enumerate}

The complete subcase analysis for $a = 1$ is given in Section
\ref{case.dist.=1}. It shows that any integral maximal
lattice-free simplex in $\R^3$ with $a = 1$ satisfies $c \leq 8$ and
$f \leq 16$ or is unimodularly transformable to a simplex with $a \geq
2$.
Thus, in both cases $a \geq 2$ and $a = 1$ there is only a finite
number of potential simplices that need to be checked. After ruling
out simplices which are equivalent by unimodular transformation only
seven different simplices remain.

\begin{theorem}
Any integral maximal lattice-free simplex in dimension three can be
brought by a unimodular transformation into one of the
simplices $S_1$ - $S_7$.
\end{theorem}

The convex hulls of the columns of the seven matrices listed
below represent these integral maximal lattice-free simplices.
We remark that six of these simplices are given as examples in
\cite{treutlein}.
\begin{enumerate}
  \item Simplices with vertices on the coordinate axes: \\[3mm]
    $S_1: \begin{pmatrix} 0 & 2 & 0 & 0 \\ 0 & 0 & 3 & 0 \\
            0 & 0 & 0 & 6 \end{pmatrix}, \qquad
     S_2: \begin{pmatrix} 0 & 2 & 0 & 0 \\ 0 & 0 & 4 & 0 \\
            0 & 0 & 0 & 4 \end{pmatrix}, \qquad
     S_3: \begin{pmatrix} 0 & 3 & 0 & 0 \\ 0 & 0 & 3 & 0 \\
            0 & 0 & 0 & 3 \end{pmatrix}.$
  \item Other simplices: \\[3mm]
    $S_4: \begin{pmatrix} 0 & 1 & 2 & 3 \\ 0 & 0 & 4 & 0 \\
            0 & 0 & 0 & 4 \end{pmatrix}, \qquad
     S_5: \begin{pmatrix} 0 & 1 & 2 & 3 \\ 0 & 0 & 5 & 0 \\
            0 & 0 & 0 & 5 \end{pmatrix}, \qquad
     S_6: \begin{pmatrix} 0 & 3 & 1 & 2 \\ 0 & 0 & 3 & 0 \\
            0 & 0 & 0 & 3 \end{pmatrix}, \qquad
     S_7: \begin{pmatrix} 0 & 4 & 1 & 2 \\ 0 & 0 & 2 & 0 \\
            0 & 0 & 0 & 4 \end{pmatrix}.$
\end{enumerate}

\section{Simplices in higher dimensions} \label{ext}

The case distinctions in Sections \ref{case.dist.>=2} and
\ref{case.dist.=1} are specialized to the
three dimensional geometry. In order to provide a characterization of
integral maximal lattice-free simplices in higher dimensions, it seems
unavoidable to develop a general proof technique. Although we do not
have this machinery at hand today, we believe that the number of
integral maximal lattice-free simplices is finite for any $d \in \N$.
As a first indication for the correctness of this conjecture we will
show finiteness for a special class of simplices, namely those that
have vertices on the coordinate axes.

Let $T \subseteq \R^d$ be an integral maximal lattice-free simplex
with one vertex being $0$ and the other vertices lying on the $d$
coordinate axes. Without loss of generality we assume that
$T = \conv(0,\lambda_1 e_1, \dots, \lambda_d e_d)$, where $\lambda_j
\in \Z_{>0}$ for all $j = 1, \dots, d$ and $e_j$ denotes the $j$-th
unit vector in $\R^d$. We further assume that $\lambda_1 \leq \dots
\leq \lambda_d$. In particular, we have $2 \leq \lambda_1$, otherwise
there exists a facet of $T$ which does not contain an interior integer 
point. The inequality representation of $T$ is given by the $d$
inequalities $x_j \geq 0$, $j = 1, \dots, d$, and an 
additional inequality of the form
\begin{equation} \label{d-facet}
  \alpha_1 x_1 + \dots + \alpha_d x_d \leq r,
\end{equation}
where $\alpha_j \in \N_{>0}$ for all $j = 1, \dots, d$ and $r \in
\N_{>0}$.
It may be checked that the following $14$ inequalities together with
nonnegativity define integral simplices in dimension four that are
maximal lattice-free.
\begin{alignat*}{6}
21&x_1\ + &\ 14&x_2\ + &\ 6&x_3\ + &\  &x_4 \ \leq&\ 4&2, \qquad
&(T_1)& \\
15&x_1\ + &\ 10&x_2\ + &\ 3&x_3\ + &\ 2&x_4 \ \leq&\ 3&0, \qquad
&(T_2)& \\
12&x_1\ + &\  8&x_2\ + &\ 3&x_3\ + &\  &x_4 \ \leq&\ 2&4, \qquad
&(T_3)& \\
10&x_1\ + &\  5&x_2\ + &\ 4&x_3\ + &\  &x_4 \ \leq&\ 2&0, \qquad
&(T_4)& \\
 9&x_1\ + &\  6&x_2\ + &\ 2&x_3\ + &\  &x_4 \ \leq&\ 1&8, \qquad
&(T_5)& \\
 6&x_1\ + &\  4&x_2\ + &\  &x_3\ + &\  &x_4 \ \leq&\ 1&2, \qquad
&(T_6)& \\
 6&x_1\ + &\  3&x_2\ + &\ 2&x_3\ + &\  &x_4 \ \leq&\ 1&2, \qquad
&(T_7)& \\
 5&x_1\ + &\  2&x_2\ + &\ 2&x_3\ + &\  &x_4 \ \leq&\ 1&0, \qquad
&(T_8)& \\
 4&x_1\ + &\  4&x_2\ + &\ 3&x_3\ + &\  &x_4 \ \leq&\ 1&2, \qquad
&(T_9)& \\
 4&x_1\ + &\  3&x_2\ + &\ 3&x_3\ + &\ 2&x_4 \ \leq&\ 1&2, \qquad
&(T_{10})& \\
 4&x_1\ + &\  2&x_2\ + &\  &x_3\ + &\  &x_4 \ \leq&\  &8, \qquad
&(T_{11})& \\
 3&x_1\ + &\   &x_2\ + &\  &x_3\ + &\  &x_4 \ \leq&\  &6, \qquad
&(T_{12})& \\
 2&x_1\ + &\  2&x_2\ + &\  &x_3\ + &\  &x_4 \ \leq&\  &6, \qquad
&(T_{13})& \\ 
  &x_1\ + &\   &x_2\ + &\  &x_3\ + &\  &x_4 \ \leq&\  &4. \qquad
&(T_{14})&
\end{alignat*}

\begin{theorem} \label{Th3.1}
Let $2 \leq \lambda_1 \leq \dots \leq \lambda_d$ be integers and
$T = \conv(0,\lambda_1 e_1, \dots, \lambda_d e_d) \subseteq \R^d$.

\begin{enumerate}[(a)]
  \item If $T$ is maximal lattice-free, then $\lambda_d$ is bounded by
    $\lambda_d^\ast$ which is a solution 
    to the following recursion: \\[-6mm]
\begin{align*}
\lambda_1^\ast & = 2, \\
\lambda_j^\ast & = 1 + \prod_{i=1}^{j-1}{\lambda_i^\ast},
  \qquad \forall\ j = 2, \dots, d-1, \\
\lambda_d^\ast & = \prod_{i=1}^{d-1}{\lambda_i^\ast}.
\end{align*}
  \item For $d=4$, every integral maximal lattice-free simplex 
    of the form $T$
    is defined by nonnegativity and one of the inequalities $T_1$ -
    $T_{14}$. 
\end{enumerate}

\end{theorem}

\begin{proof}
(a): 
First, consider the inequality \eqref{d-facet}. We show
that $r = \alpha_1 + \dots + \alpha_d$. Observe that $r \leq \alpha_1
+ \dots + \alpha_d$, otherwise the point $\mathbf{1} := (1, \dots, 1)
\in \R^d$ is in the interior of $T$. For the purpose of deriving a
contradiction, assume that $r < \alpha_1 + \dots + \alpha_d$.
Since $T$ is maximal lattice-free there exists an
integer point $v = (v_1, \dots, v_d) \not = \mathbf{1}$ in the
relative interior of the facet $T \cap \{x \in \R^d: \alpha_1 x_1 +
\dots + \alpha_d x_d = r\}$. In particular, $v$
satisfies $v_j \geq 1$ for all $j = 1, \dots, d$ and $v_j > 1$ for at
least one $j \in \{1, \dots, d\}$. However, since $\alpha_j > 0$ for
all $j = 1, \dots, d$ this implies $r = \alpha_1
v_1 + \dots + \alpha_d v_d > \alpha_1 + \dots + \alpha_d$ which is
a contradiction. Thus, we have $r = \alpha_1 + \dots + \alpha_d$.

From the definition of $T$, it follows that $\alpha_j \lambda_j =
\alpha_1 + \dots + \alpha_d$ for all $j = 1, \dots, d$ which implies
that $\alpha_1 + \dots + \alpha_d = \frac{1}{\lambda_1}(\alpha_1 +
\dots + \alpha_d) + \dots + \frac{1}{\lambda_d}(\alpha_1 +\dots +
\alpha_d)$. Hence, we obtain
\begin{equation} \label{lambda.cond.}
1 = \frac{1}{\lambda_1} + \dots + \frac{1}{\lambda_d}.
\end{equation}
We look for an upper bound $\lambda_d^\ast$ on $\lambda_d$. Due to
\eqref{lambda.cond.}, the coefficient $\lambda_d$ is the
bigger the smaller the other coefficients $\lambda_1, \dots,
\lambda_{d-1}$ are.
It follows from an inductive argument that, starting with
$\lambda_1$, gradually fixing the $\lambda_j$'s at their minimal
possible value yields the maximal value for $\lambda_d$.
In this way, we can recursively compute $\lambda_d$.
Again, this recursion can be formally shown using an inductive
argument. Thus, $\lambda_d$ is bounded.

(b): For $d=4$ the recursion yields $\lambda_4 \leq 42$.
By enumeration, we obtain the $14$ simplices defined by
nonnegativity and one of the inequalities $T_1$ - $T_{14}$.
\end{proof}
\bigskip

The key element of the proof of Theorem \ref{Th3.1} is a recursive
formula. Next we illustrate this recursion for the case of $d=5$.

\begin{example}
First, $\lambda_1$
is fixed at $2$. Then, \eqref{lambda.cond.} implies that $\frac{1}{2}
= \frac{1}{\lambda_2} + \frac{1}{\lambda_3} + \frac{1}{\lambda_4} +
\frac{1}{\lambda_5}$. Thus, the choice $\lambda_2 = 3$ is minimal and
substituting $(\lambda_1,\lambda_2) = (2,3)$ in \eqref{lambda.cond.}
yields $\frac{1}{6} = \frac{1}{\lambda_3} + \frac{1}{\lambda_4} +
\frac{1}{\lambda_5}$. The minimal possible value for $\lambda_3$ is
now $7$ and from \eqref{lambda.cond.} we obtain $\frac{1}{42} =
\frac{1}{\lambda_4} + \frac{1}{\lambda_5}$. The last step is to fix
$\lambda_{4} = 43$ which leads to $\lambda_5 = 1806$. $\hfill \diamond$
\end{example}

Observe that the above recursion does not only give an upper bound on
$\lambda_d$, but also constructs an integral maximal lattice-free
simplex with the largest possible value for $\lambda_d$. For instance,
if $d=5$, $T = \conv(2 e_1, 3 e_2, 7 e_3, 43 e_4, 1806 e_5)$ is
maximal lattice-free.
\bigskip

As a second indication why in any dimension $d$ there should only be a
finite number of integral maximal lattice-free simplices, we note that
the follwing property holds. For any integral maximal lattice-free
simplex $T$, let us denote by ${\cal F}_1, \dots, {\cal F}_{d+1}$ its
facets.
Each facet ${\cal F}_i$ contains interior integer points. In fact, we
can search for a maximal sublattice fully contained in the interior of
${\cal F}_i$ that is unimodularly transformable to some $\Z^{s_i}$ for
$s_i \in \{0, \dots, d-1\}$.

\begin{observation}
We have  $s_i = d-1$
for at most one $i$.
\end{observation}

\begin{proof}
For the purpose of deriving a contradiction, assume that $s_i = s_j =
d-1$ for two facets ${\cal F}_i$ and ${\cal F}_j$ with $i \not = j$.
Then, ${\cal F}_i$ and ${\cal F}_j$ have each at least $2^{d-1}$
integer points of different parity in their interior.
If two integer points $w_i$ and $w_j$ on ${\cal F}_i$ and 
${\cal F}_j$, respectively, have the same parity, then
$\frac{1}{2}(w_i + w_j)$ is an interior point in $T$. Hence, we have
$2^d$ integer points with different parity in the interior of
${\cal F}_i \cup {\cal F}_j$. Now, any interior integer point of a
facet different from ${\cal F}_i$ and ${\cal F}_j$ will lead to a
contradiction. 
\end{proof}
\bigskip

This shows that the possibilities for the sublattice structure in the
interior of the facets is somehow limited.
Finally, let us remark, that if $s_i = 0$ for all $i$, finiteness
follows from a result of Lagarias and Ziegler \cite{lazi}.


\section{Details on the case distinction for $\boldsymbol{a \geq 2}$} \label{case.dist.>=2}

Let $S \subseteq \R^3$ be an integral maximal lattice-free simplex
with $a \geq 2$ given by the following inequality description:
\begin{alignat}{3}
     &    &            &    &
     -\ &x_3 \leq 0, \label{facet.1} \\
     &    &        -\ f&x_2 &
     +\ e&x_3 \leq 0, \label{facet.2} \\
  -cf&x_1 &       +\ bf&x_2 &
     +\ (cd-be)&x_3 \leq 0, \label{facet.3} \\
   cf&x_1 & \ +\ f(a-b)&x_2 &
     \ +\ \big(e(b-a) + c(a-d)\big)&x_3 \leq acf. \label{facet.4}
\end{alignat}
In this section, we will often state interior integer
points as counterexamples and simply prove that for those points all
the restrictions \eqref{facet.1} - \eqref{facet.4} are satisfied with
a strict inequality.
Recall, that the unknowns $a,b,c,d,e$, and $f$ are integer and
satisfy the following properties:
\begin{alignat*}{2}
2 & \leq a, \qquad & \qquad & 0 \leq b < c, \\
2 & \leq c, \qquad & \qquad & 0 \leq d < f, \\
2 & \leq f, \qquad & \qquad & 0 \leq e < f.
\end{alignat*}
We make frequently use of the following simple observation.

\begin{observation} \label{obs} \quad

\begin{enumerate}[(a)]
  \item \label{obs1} Let $2 \leq l \in \N$. Then, for any integer
    $x \geq l$, we have $\frac{x}{x-1} \leq \frac{l}{l-1}$. \\[-6mm]
  \item \label{obs2} Let $2 \leq l_x, l_y \in \N$. Then, for any
    integers $x \geq l_x$ and $y \geq l_y$, we have $\frac{xy}{x+y}
    \geq \frac{l_x l_y}{l_x + l_y}$. \\[-6mm]
  \item \label{obs3} Let $2 \leq l_x \in \N$ and $3 \leq l_y \in \N$.
    Then, for any integers $x \geq l_x$ and $y \geq l_y$, we have
    $\frac{xy}{xy-x-y} \leq \frac{l_x l_y}{l_x l_y - l_x - l_y}$.
\end{enumerate}

\end{observation}

\subsection*{I) $\boldsymbol{cf + f(a-b) + e(b-a) + c(a-d) \geq acf}$}

\subsubsection*{1) $\boldsymbol{b \geq a}$}

In this case, we have $acf \leq cf + (f-e)(a-b) + c(a-d) \leq ac + cf$
which implies that $af \leq a + f$. Since $a \geq 2$ and $f \geq 2$
this inequality is only satisfied for $(a,f) = (2,2)$. Substituting
this in $cf + f(a-b) + e(b-a) + c(a-d) \geq acf$
yields $4c \leq 2c + 2(2-b) + e(b-2) + c(2-d) \Leftrightarrow cd \leq
(2-e)(2-b)$. Since $b \geq a = 2$ and $e < f = 2$ we obtain $cd \leq
0$ and therefore $d = 0$. If $b > 2$ then $0 = cd \leq (2-e)(2-b) < 0$
is a contradiction. Hence, we have $b = 2$. However, $S$ is now
contained in $\conv\left((0,0,0)^T,(2,0,0)^T,(0,0,2)^T\right) +
\spann(e_2)$ which is a contradiction to its maximality.

\subsubsection*{2) $\boldsymbol{b < a}$}

\paragraph*{i) $\boldsymbol{(a,c) = (2,2)}$}

If $b = 0$ then there is no interior integer point in the facet
$0,v^1,v^2$. Hence, $b = 1$. Substituting $(a,b,c) = (2,1,2)$ in $cf +
f(a-b) + e(b-a) + c(a-d) \geq acf$ yields $4f \leq 2f + f - e + 2(2-d)
\leq 3f + 4$ and therefore $f \leq 4$.

\paragraph*{ii) $\boldsymbol{(a,c) \not = (2,2)}$}

Here, we obtain $acf \leq cf + f(a-b) + e(b-a) + c(a-d) \leq ac + af +
cf$ which implies
\begin{equation} \label{U1}
  f(ac - a - c) \leq ac.
\end{equation}
From \eqref{U1} and Observation \ref{obs}(\ref{obs3}), it follows that
$f \leq \frac{ac}{ac - a - c} \leq \frac{2 \cdot 3}{2 \cdot 3 - 2 - 3}
= 6$.

Now assume $(a,f) = (2,2)$. Then $c \geq 3$ and $b,d,e \in
\{0,1\}$. Substituting $(a,f) = (2,2)$ in $cf + f(a-b) + e(b-a) +
c(a-d) \geq acf$ implies that $2c + 2(2-b) + e(b-2) + c(2-d) \geq 4c
\Leftrightarrow cd \leq (2-e)(2-b)$. If $b = 1$ or $e = 1$ it follows
$cd \leq 2$ and therefore $d = 0$. Hence, the facet $0,v^1,v^3$ does
not contain an interior integer point. Thus, let $b = e = 0$. We must
have $d \not = 0$ since otherwise the facet  $0,v^1,v^3$ does not
contain an interior integer point. This implies $d = 1$ and
$c = cd \leq 4$.

Therefore, we can assume that $(a,f) \not = (2,2)$. From \eqref{U1}
and Observation \ref{obs}(\ref{obs3}), it follows that $c \leq
\frac{af}{af - a - f} \leq 6$ and it remains to find an upper bound on
$a$. If $(c,f) = (2,2)$ there is no interior integer point in the
facet $0,v^2,v^3$. On the other hand, if $(c,f) \not = (2,2)$ we have
$a \leq \frac{cf}{cf - c - f} \leq 6$ by \eqref{U1} and Observation
\ref{obs}(\ref{obs3}).

\subsection*{II) $\boldsymbol{cf + f(a-b) + e(b-a) + c(a-d) < acf}$}

For the purpose of deriving a contradiction assume that $-cf + bf + cd
- be < 0$. In this case, the point $(1,1,1)$ is in the interior of
$S$ as one can easily check by substituting $(1,1,1)$ in the
inequalities \eqref{facet.1}-\eqref{facet.4}: in all four
inequalities, the left hand side is strictly less than the right hand
side.
We therefore must have
\begin{equation} \label{Ineq.1}
  -cf + bf + cd - be \geq 0.
\end{equation}
We first show that $k \geq 1$.
Note that $ac - a - c + b \leq 0$ holds true only for $(a,b,c) =
(2,0,2)$. However, then the facet $0,v^1,v^2$ does not contain an
interior integer point. Thus, we have $ac - a - c + b > 0 
\Leftrightarrow \frac{ac+b-a}{c} > 1$ which implies that $k \geq 1$.

\subsubsection*{1) $\boldsymbol{-cfk + bf + cd - be \geq 0}$}

We have already shown that $k \geq 1$.
Assume $k \geq 2$. Then $-cfk
+ bf + cd - be \leq -2cf + bf + cd - be = f(b-c) + c(d-f) - be < 0$
which is a contradiction. Hence, $k = 1$ and it follows
\begin{equation} \label{U2}
  1 < \frac{ac + b - a}{c} \leq 2.
\end{equation}
From \eqref{U2}, we obtain $a(c-1) \leq 2c - b \leq 2c$
which implies $a \leq \frac{2c}{c-1} \leq 4$ by Observation
\ref{obs}(\ref{obs1}).

Moreover, we have $b > 0$ since otherwise $-cf + bf + cd - be = c(d-f)
< 0$ is a contradiction. From \eqref{U2}, it
follows $b-a \leq c(2-a) \leq 0$ which means that $b \leq a$. From
\eqref{Ineq.1}, we know that $c \leq b \frac{f-e}{f-d} \leq af$.
 
Assume $a = 4$. Then, \eqref{U2} implies $2c + b
\leq 4$ which can never hold true for $c \geq 2$ and $b > 0$. Thus, $a
\not = 4$. So assume $a = 3$. Now, \eqref{U2}
implies $c + b \leq 3$ which is only satisfied for $(b,c) =
(1,2)$. Substituting this in \eqref{Ineq.1} yields $f \leq 2d -
e$. For $d \geq 4$ the point $(2,1,1)$ is in the interior of $S$:
obviously, \eqref{facet.1} and \eqref{facet.2} are strict;
\begin{alignat*}{2}
& \eqref{facet.3}: \quad & -2cf + bf + cd - be =
  f(b-c) + c(d-f) - be &< 0; \\
& \eqref{facet.4}: \quad & 2cf + f(a-b) + e(b-a) + c(a-d) &= \\
& & 4f + 2f - 2e + 6 - 2d = 6f + 2(3 - d -e) < 6f &= acf.
\end{alignat*}
So we have $d \leq 3$ which implies $f \leq 2d - e \leq 6$ and
$c \leq af \leq 18$. It remains to consider the case $a = 2$.

\paragraph*{i) $\boldsymbol{(a,c) = (2,2)}$}

From $0 < b < c = 2$, it follows that $b = 1$ and \eqref{Ineq.1}
implies that $f \leq 2d - e$. First note that $e \leq 2$ since
otherwise the point $(2,1,1)$ is in the interior of $S$:
clearly, \eqref{facet.1} and \eqref{facet.2} are strict;
\begin{alignat*}{2}
& \eqref{facet.3}: \quad & -2cf + bf + cd - be =
  f(b-c) + c(d-f) - be &< 0; \\
& \eqref{facet.4}: \quad & 2cf + f(a-b) + e(b-a) + c(a-d) &= \\
& & 4f + f - e + 4 - 2d \leq 4f - 2e + 4 < 4f &= acf.
\end{alignat*}
Now assume $f \geq 9$. If $f + 4 < 2d + e$ holds true, then the point
$(2,1,1)$ is in the interior of $S$: as above, \eqref{facet.1},
\eqref{facet.2}, and \eqref{facet.3} are strict;
\begin{alignat*}{2}
& \eqref{facet.4}: \quad & 2cf + f(a-b) + e(b-a) + c(a-d) =
4f + f - e + 4 - 2d < 4f &= acf.
\end{alignat*}
If $f + 4 \geq 2d + e$, the point $(2,1,2)$ is in the interior of $S$:
clearly, \eqref{facet.1} is strict; \eqref{facet.2} is strict since $f
\geq 9$ and $e \leq 2$;
\begin{alignat*}{2}
& \eqref{facet.3}: \quad & -2cf + bf + 2cd - 2be =
  -4f + f + 4d - 2e \leq -f + 8 - 4e &< 0;\\ 
& \eqref{facet.4}: \quad & 2cf + f(a-b) + 2e(b-a) + 2c(a-d) &= \\
& & 4f + f - 2e + 8 - 4d \leq 4f - f - 4e + 8 < 4f &= acf.
\end{alignat*}
Hence, we must have $f \leq 8$.

\paragraph*{ii) $\boldsymbol{(a,c) \not = (2,2)}$}

In this case, we have $c \geq 3$. If $f \geq 7$ the point $(2,1,1)$ is
in the interior of $S$: clearly, \eqref{facet.1},\eqref{facet.2} and
\eqref{facet.3} are strict (see above);
\begin{alignat*}{2}
& \eqref{facet.4}: \quad & 2cf + f(a-b) + e(b-a) + c(a-d) &= \\
& & cf + cf - bf - cd + be + 2(c + f - e) &\leq \\
& & cf + 2(c + f - e) < cf + cf = 2cf &= acf.
\end{alignat*} 
Here, the strict inequality in the last row follows from the fact that
$c + f - e < \frac{1}{2}cf$ for $c \geq 3$ and $f \geq 7$ (use
Observation \ref{obs}(\ref{obs2})). Therefore, we obtain $f \leq 6$
and it follows $c \leq af \leq 12$.

\subsubsection*{2) $\boldsymbol{-cfk + bf + cd - be < 0}$}

For the purpose of deriving a contradiction assume
$cfk + f(a-b) + e(b-a) + c(a-d) < acf$. Then, the point $(k,1,1)$ is
in the interior of $S$. One can see this by substituting $(k,1,1)$ in
the inequalities \eqref{facet.1}-\eqref{facet.4}. Thus, it follows
\begin{equation} \label{Ineq.2}
cfk + f(a-b) + e(b-a) + c(a-d) \geq acf.
\end{equation}
We know that $k \geq 1$. Now assume $k = 1$. From \eqref{Ineq.1} and
\eqref{Ineq.2} it follows that $acf \leq cf + f(a-b) + e(b-a) +
c(a-d) \leq a(c + f - e)$. Note that $c + f- e < cf$ holds true for
any feasible triple $(c,e,f) \not = (2,0,2)$ and would lead to a
contradiction in this chain of inequalities. Therefore, we must have
$(c,e,f) = (2,0,2)$. However, in this case $S$ is contained in
$\conv\left((0,0,0)^T,(0,2,0)^T,(0,0,2)^T\right) + \spann(e_1)$ which
contradicts the maximality of $S$. Thus, we have $k \geq 2$.

From \eqref{Ineq.1} it follows that $cd - be \geq f(c-b) > 0$ and from
\eqref{Ineq.2} it follows that $e(b-a) + c(a-d) \geq f(-ck
+ ac + b - a) > f(-c\frac{ac + b - a}{c} + ac + b - a) = 0$ . Thus, we
have
\begin{equation} \label{U3}
cd - be > 0
\end{equation}
and
\begin{equation} \label{U4}
e(b-a) + c(a-d) > 0.
\end{equation}

\paragraph*{i) $\boldsymbol{e > 0}$}

Here, the point $(d,e,1)$ is in the interior of $S$: obviously,
\eqref{facet.1} is strict;  \eqref{facet.2} is strict since we have
$e > 0$;
\begin{alignat*}{2}
& \eqref{facet.3}: \quad & -cfd + bfe + cd - be =
(1-f)(cd - be) &< 0; \\
& \eqref{facet.4}: \quad & cfd + f(a-b)e + e(b-a) + c(a-d) &= \\
& & (1-f)(-cd + e(b-a)) + ac < (1-f)(-ac) + ac &= acf.
\end{alignat*}
The strict inequalities follow from \eqref{U3} and \eqref{U4}.

\paragraph*{ii) $\boldsymbol{e = 0}$}

By \eqref{U3} and \eqref{U4}, we obtain $0 < d < a$. Furthermore,
\eqref{Ineq.1} and \eqref{Ineq.2} change to
\begin{equation} \label{Ineq.3}
-cf + bf + cd \geq 0
\end{equation}
and
\begin{equation} \label{Ineq.4}
cfk + f(a-b) + c(a-d) \geq acf.
\end{equation}

\subparagraph*{A) $\boldsymbol{a = b}$}

In this case we have $k=a-1$. Substituting this in \eqref{Ineq.4}
yields $d \leq a - f$ and since $d > 0$ this
implies that $f < a$. If $d \geq 2$ the point $(2,1,1)$ is in the
interior of $S$: evidently, \eqref{facet.1} and \eqref{facet.2} are
strict;
\begin{alignat*}{2}
& \eqref{facet.3}: \quad & -2cf + bf + cd - be =
f(b-c) + c(d-f) &< 0; \\
& \eqref{facet.4}: \quad & 2cf + f(a-b) + e(b-a) + c(a-d) &= \\ 
& & 2cf + c(a-d) < 2cf + cf(a-d) = acf + cf(2-d) &\leq acf.
\end{alignat*}
So we have $d = 1$.
If $a \geq 4$ we have $(2,1,1)$ in the interior of $S$: as above,
\eqref{facet.1}, \eqref{facet.2}, and \eqref{facet.3} are strict;
\begin{alignat*}{2}
& \eqref{facet.4}: \quad & 2cf + f(a-b) + e(b-a) + c(a-d) =
2cf + c(a-1) < 2cf + cf(a-2) &= acf.
\end{alignat*}
The strict inequality follows from Observation \ref{obs}(\ref{obs1}):
$\frac{a-1}{a-2} \leq \frac{3}{2} < 2 \leq f$ since $a \geq 4$.
Thus, let $a \leq 3$.
The chain $0 < d < f < a \leq 3 $ implies that $(a,f) =
(3,2)$. Substituting this in \eqref{Ineq.3} yields $c \leq 6$.

\subparagraph*{B) $\boldsymbol{a < b}$}

If $d \geq 2$, the point $(2,1,1)$ is in the interior of $S$: as
above, \eqref{facet.1}, \eqref{facet.2}, and \eqref{facet.3} are
strict;
\begin{alignat*}{2}
& \eqref{facet.4}: \quad & 2cf + f(a-b) + e(b-a) + c(a-d) &< \\
& & 2cf + c(a-d) < 2cf + cf(a-d) = acf + cf(2-d) &\leq acf.
\end{alignat*}
Hence, $d = 1$. From \eqref{Ineq.3} and Observation
\ref{obs}(\ref{obs1}), it follows that
\begin{equation} \label{U5}
c \leq b\frac{f}{f-1} \leq 2b.
\end{equation}
If $a \geq 3$, the point $(2,1,1)$ is in the interior of $S$: as
above, \eqref{facet.1}, \eqref{facet.2}, and \eqref{facet.3} are
strict;
\begin{alignat*}{2}
& \eqref{facet.4}: \quad & 2cf + f(a-b) + e(b-a) + c(a-d) <
2cf + c(a-1) \leq 2cf + cf(a-2) &= acf.
\end{alignat*}
The last inequality follows from Observation \ref{obs}(\ref{obs1}):
$\frac{a-1}{a-2} \leq  2 \leq f$ since $a \geq 3$.
So let $a = 2$. This implies $b \geq
3$. Furthermore, it follows $2 \leq k < \frac{ac + b - a}{c} = 
2\frac{c-1}{c} + \frac{b}{c} < 3$ and therefore we obtain that $k =
2$. From \eqref{Ineq.4} it follows that
$2cf + f(2-b) + c \geq 2cf$ which implies
\begin{equation} \label{U6}
c \geq f(b-2) \geq f.
\end{equation}
Now consider the two inequalities $c \leq b\frac{f}{f-1}$  and $c \geq
f(b-2)$ arising from \eqref{U5} and \eqref{U6}. If $f \geq 3$ it
follows $3(b-2) \leq f(b-2) \leq c \leq b\frac{f}{f-1} \leq
\frac{3}{2}b \Rightarrow b \leq 4$. This implies $c \leq 4 \cdot
\frac{3}{2} = 6$ and $f \leq c \leq 6$.
It remains to consider the case $f = 2$. Here, \eqref{U5} and
\eqref{U6} imply $2b - 4 \leq c \leq 2b$. If $b \geq 7$, the point
$(2,2,1)$ is in the interior of $S$: clearly, \eqref{facet.1} and
\eqref{facet.2} are strict;
\begin{alignat*}{2}
& \eqref{facet.3}: \quad & -2cf + 2bf + cd - be =
4b - 3c \leq 4b + 12 - 6b = 12 - 2b &< 0; \\ 
& \eqref{facet.4}: \quad & 2cf + 2f(a-b) + e(b-a) + c(a-d) =
4c + 4(2-b) + c \leq 4c + 8 - 2b < 4c &= acf.
\end{alignat*}
Thus, $b \leq 6$ and it follows $c \leq 2b \leq 12$. 

\subparagraph*{C) $\boldsymbol{a > b}$}

First we argue that $a \not = 2$. Assuming that $a = 2$, we obtain
that $k \leq a - 1 = 1$. This contradicts $k \geq 2$. Hence, let $a
\geq 3$. If $a \geq 7$, the point $(2,1,1)$ is in the interior of $S$:
as above, \eqref{facet.1}, \eqref{facet.2}, and \eqref{facet.3} are
strict;
\begin{alignat*}{2}
& \eqref{facet.4}: \quad & 2cf + f(a-b) + e(b-a) + c(a-d) &= \\
& & cf + ac + af + cf - bf - cd \leq ac + af + cf &< acf.
\end{alignat*}
The strict inequality follows from the fact that $\frac{a}{a-1}
\leq \frac{7}{6} < \frac{6}{5} \leq \frac{cf}{c+f}$ by Observation
\ref{obs}(\ref{obs1}) and \ref{obs}(\ref{obs2}) as $a \geq 7$ and
as we cannot
have $(c,f) = (2,2)$ since in this case $S$ is contained in
$\conv\left((0,0,0)^T,(0,2,0)^T,(0,0,2)^T\right) + \spann(e_1)$ which 
contradicts the maximality of $S$. Thus, $a \leq 6$.

Next we show that it is impossible for $c$  to be greater or equal to
$7$. If $c \geq 7$, then the point $(2,1,1)$ would be in
the interior of $S$: as above, \eqref{facet.1}, \eqref{facet.2}, and
\eqref{facet.3} are strict;
\begin{alignat*}{2}
& \eqref{facet.4}: \quad & 2cf + f(a-b) + e(b-a) + c(a-d) &= \\
& & cf + ac + af + cf - bf - cd \leq ac + af + cf &< acf.
\end{alignat*}
Here, the strict
inequality follows from the fact that $\frac{c}{c-1} \leq \frac{7}{6} <
\frac{6}{5} \leq \frac{af}{a+f}$ by Observation
\ref{obs}(\ref{obs1}) and \ref{obs}(\ref{obs2}) as $ c \geq 7$ and as
$(a,f) \geq (3,2)$.
Thus, $c \leq 6$. Similarly, it can be verified that we must
have $f \leq 6$ as otherwise the point $(2,1,1)$ is in the interior of
$S$.

\section{Details on the case distinction for $\boldsymbol{a = 1}$} \label{case.dist.=1}

Let $S \subseteq \R^3$ be an integral maximal lattice-free simplex
with $a = 1$ given by the following inequality description:
\begin{alignat}{3}
     &    &            &    &
     -\ &x_3 \leq 0, \label{F.1} \\
     &    &        -\ f&x_2 &
     +\ e&x_3 \leq 0, \label{F.2} \\
  -cf&x_1 &       +\ bf&x_2 &
     +\ (cd-be)&x_3 \leq 0, \label{F.3} \\
   cf&x_1 & \ +\ f(1-b)&x_2 &
     \ +\ \big(e(b-1) + c(1-d)\big)&x_3 \leq cf. \label{F.4}
\end{alignat}
As in Section \ref{case.dist.>=2}, we will often state interior integer
points as counterexamples by proving that for those points all the
restrictions \eqref{F.1} - \eqref{F.4} are satisfied with a strict
inequality.
Recall, that the unknowns $b,c,d,e$, and $f$ are integer and
satisfy the following properties:
\begin{alignat*}{2}
2 & \leq c, \qquad & \qquad & 0 \leq b < c, \\
2 & \leq f, \qquad & \qquad & 0 \leq d < f, \\
  &         \qquad & \qquad & 0 \leq e < f.
\end{alignat*}
First note that $b > 1$ since otherwise there is no interior integer
point in the facet spanned by the three points $(0,0,0)^T$,
$(1,0,0)^T$, and $(b,c,0)^T$. This implies $b \geq 2$ and $c \geq 3$.

\subsection*{I) $\boldsymbol{cf + f(1-b) + e(b-1) + c(1-d) \geq cf}$}

From $cf + f(1-b) + e(b-1) + c(1-d) \geq cf$, it follows that $c(d-1)
\leq (f-e)(1-b) < 0$. Thus, we obtain $d = 0$. However, there is no
interior integer point in the facet spanned by the three points
$(0,0,0)^T$, $(1,0,0)^T$, and $(0,e,f)^T$ which is a contradiction.

\subsection*{II) $\boldsymbol{cf + f(1-b) + e(b-1) + c(1-d) < cf}$}

In this case, we must have $-cf + bf + cd - be \geq 0$ since otherwise
the point $(1,1,1)$ is in the interior of $S$ as one easily checks by
substituting $(1,1,1)$ in the inequalities
\eqref{F.1}-\eqref{F.4}. Thus, we obtain
\begin{equation} \label{(1,1,1)}
  -cf + bf + cd - be \geq 0.
\end{equation}
For the purpose of deriving a contradiction assume that $cf + f(1-b) +
e(b-1) + c(1-d) < 0$. Then, the point $(2,1,1)$ is in the interior of
$S$: clearly, \eqref{F.1}, \eqref{F.2}, and \eqref{F.4} are strict.
\begin{alignat*}{2}
& \eqref{F.3}: \quad & -2cf + bf + cd - be = f(b-c) + c(d-f) - be < 0.
\end{alignat*}
Hence, it follows
\begin{equation} \label{(2,1,1)}
  cf + f(1-b) + e(b-1) + c(1-d) \geq 0.
\end{equation}
Assume $d \leq e$. Using \eqref{(1,1,1)} yields $0 \leq -cf + bf +
cd - be \leq -cf + bf + ce - be = (b-c)(f-e) < 0$ which is a
contradiction. Therefore, it holds $d > e$. 

We now construct a sequence of points which helps to derive conditions
for the unknown variables. These conditions are used later in the
subcase analysis.
For the moment assume $c-e \leq f$ and consider the sequence of points
$(\theta,\theta,1)$, where $\theta \geq 1$. By equation
\eqref{(1,1,1)} and the relation $c > b$, we obtain that there exists
some $\Theta \geq 2$ such that
\begin{alignat}{2}
 -cf(\Theta - 1) + bf(\Theta - 1) + cd - be \geq 0&, \label{sequ.1-1} \\
 -cf\Theta + bf\Theta + cd - be < 0&. 
\end{alignat}
Since the point $(\Theta,\Theta,1)$ satisfies \eqref{F.1},
\eqref{F.2}, and  \eqref{F.3} strictly, we must have that
\begin{equation} \label{sequ.1-3}
  cf\Theta + f(1-b)\Theta + e(b-1) + c(1-d) \geq cf
\end{equation}
since otherwise the point $(\Theta,\Theta,1)$ is in the interior of
$S$. Adding \eqref{sequ.1-1} and \eqref{sequ.1-3} yields
\begin{equation} \label{sequ.1-4}
  f(\Theta-b) + c - e \geq 0.
\end{equation}
Using \eqref{(2,1,1)} and \eqref{sequ.1-1} together with our assumption
$c-e \leq f$, we obtain $f(\Theta-1)(c-b) \leq cd -be \leq f(c-b) + f
+ c - e \leq f(c-b) + 2f$ and hence $f(\Theta-1)(c-b) \leq f(c-b) + 2f
\Leftrightarrow \Theta \leq 2 + \frac{2}{c-b}$. We infer
\begin{equation} \label{bound.on.b}
  \Theta \leq
  \begin{cases}
    4, & \textup{if } c = b+1 \\
    3, & \textup{if } c = b+2 \\
    2, & \textup{if } c \geq b+3,
  \end{cases} \quad \stackrel{\eqref{sequ.1-4}}{\Longrightarrow} \quad
  b \leq \Theta + \frac{c-e}{f} \leq \Theta + 1 \leq
  \begin{cases}
    5, & \textup{if } c = b+1 \\
    4, & \textup{if } c = b+2 \\
    3, & \textup{if } c \geq b+3.
  \end{cases}
\end{equation}
This shows that $2 \leq b \leq 5$ whenever $c-e \leq f$ holds true.

\subsubsection*{1) $\boldsymbol{e = 0}$}

Since $e = 0$ the inequalities \eqref{(1,1,1)} and \eqref{(2,1,1)}
change to 
\begin{equation} \label{point.1}
-cf + bf + cd \geq 0
\end{equation}
and
\begin{equation} \label{point.2}
cf + f(1-b) + c(1-d) \geq 0.
\end{equation}
Furthermore, we can assume without loss of generality that $c \leq
f$. Otherwise, if $c > f$, we switch coordinates by applying the
unimodular transformation to $S$ with $M$ being the identity matrix
in $\R^3$ where the last two columns are interchanged and $v$ being
$0$ (see Definition \ref{unimod.trans}).

By assumption, $c - e \leq f$ is now satisfied and thus, by
\eqref{bound.on.b}, the variable $b$ is bounded. In addition, $c$ is
bounded for $b = 4$ and $b = 5$. We first find an upper bound on $c$
for the remaining cases where $b = 2$ and $b = 3$.
\bigskip

Let $b = 2$. We show that $c \leq 7$. So assume $c \geq 8$. If
$f(2c-3) + c(1-d) < cf$ holds true, then the point $(2,3,1)$ is in the
interior of $S$: clearly, \eqref{F.1} and \eqref{F.2} are strict;
\begin{alignat*}{2}
& \eqref{F.3}: \quad & -2cf + 3bf + cd - be = -2cf + 6f + cd
= f(6-c) + c(d-f) &< 0; \\ 
& \eqref{F.4}: \quad & 2cf + 3f(1-b) + e(b-1) + c(1-d) &= \\
& & 2cf - 3f + c(1-d) = f(2c-3) + c(1-d) &< cf.
\end{alignat*}
Thus, let $f(2c-3) + c(1-d) \geq cf \Leftrightarrow f \geq
\frac{c(d-1)}{c-3}$. From \eqref{point.1}, it follows that $f \leq
\frac{cd}{c-2}$. Putting this together we have $\frac{c(d-1)}{c-3}
\leq \frac{cd}{c-2} \Leftrightarrow c \geq d+2$. We infer that the
point $(2,4,1)$ is in the interior of $S$: clearly, \eqref{F.1} and
\eqref{F.2} are strict;
\begin{alignat*}{2}
& \eqref{F.3}: \quad & -2cf + 4bf + cd - be = -2cf + 8f + cd
= f(8-c) + c(d-f) &< 0; \\ 
& \eqref{F.4}: \quad & 2cf + 4f(1-b) + e(b-1) + c(1-d) =
2cf - 4f + c(1-d) &= \\
& & 2f(c-2) + c(1-d) \leq 2cd + c(1-d) =
c(d+1) \leq c(c-1) \leq c(f-1) &< cf.
\end{alignat*}
Here, the inequalities in the last row follow from \eqref{point.1},
$d+2 \leq c$, and $c \leq f$. Hence, we must have $c \leq 7$.
\bigskip

Let $b = 3$. We show that $c \leq 8$. So assume $c \geq 9$. If
$2f(c-2) + c(1-d) < cf$ holds true, then the point $(2,2,1)$ is in
the interior of $S$: clearly, \eqref{F.1} and \eqref{F.2} are strict;
\begin{alignat*}{2}
& \eqref{F.3}: \quad & -2cf + 2bf + cd - be = -2cf + 6f + cd
= f(6-c) + c(d-f) &< 0; \\ 
& \eqref{F.4}: \quad & 2cf + 2f(1-b) + e(b-1) + c(1-d) &= \\
& & 2cf - 4f + c(1-d) = 2f(c-2) + c(1-d) &< cf.
\end{alignat*}
Thus, let $2f(c-2) + c(1-d) \geq cf \Leftrightarrow f \geq
\frac{c(d-1)}{c-4}$. From \eqref{point.1}, it follows that $f \leq
\frac{cd}{c-3}$. Putting this together we have $\frac{c(d-1)}{c-4}
\leq \frac{cd}{c-3} \Leftrightarrow c \geq d+3$. We infer that the
point $(2,3,1)$ is in the interior of $S$: clearly, \eqref{F.1} and
\eqref{F.2} are strict;
\begin{alignat*}{2}
& \eqref{F.3}: \quad & -2cf + 3bf + cd - be = -2cf + 9f + cd
= f(9-c) + c(d-f) &< 0; \\ 
& \eqref{F.4}: \quad & 2cf + 3f(1-b) + e(b-1) + c(1-d) =
2cf - 6f + c(1-d) &= \\
& & 2f(c-3) + c(1-d) \leq 2cd + c(1-d) =
c(d+1) \leq c(c-2) \leq c(f-2) &< cf.
\end{alignat*}
Here, the inequalities in the last row follow from \eqref{point.1},
$d+3 \leq c$, and $c \leq f$. Hence, we must have $c \leq 8$.

It follows, that $13$ choices of $(b,c)$ are left: $(2,3)$, $(3,4)$,
$(4,5)$, $(5,6)$, $(2,4)$, $(3,5)$, $(4,6)$, $(2,5)$, $(2,6)$,
$(2,7)$, $(3,6)$, $(3,7)$, $(3,8)$. In the following we will prove
upper bounds on $f$ for each of the $13$ possibilities.
\bigskip

$\bullet$ Let $(b,c) = (2,3)$. We show that $f \leq 9$. So assume $f \geq
10$. From \eqref{point.1} and \eqref{point.2} we obtain $f \leq 3d$
and $2f \geq 3(d-1)$. If $d < \frac{4}{9}f$ holds true, then the point
$(2,1,3)$ is in the interior of $S$: clearly, \eqref{F.1} and
\eqref{F.2} are strict;
\begin{alignat*}{2}
& \eqref{F.3}: \quad & -2cf + bf + 3(cd - be) = -6f + 2f + 9d
< -4f + 4f &= 0; \\ 
& \eqref{F.4}: \quad & 2cf + f(1-b) + 3(e(b-1) + c(1-d)) &= \\
& & 6f - f + 9(1-d) < 3f + 2f + 9 - 3f &< 3f.
\end{alignat*}
Thus, let $d \geq \frac{4}{9}f$. If $d < \frac{2}{3}f$ holds true,
then the point $(2,1,2)$ is in the interior of $S$: clearly,
\eqref{F.1} and \eqref{F.2} are strict;
\begin{alignat*}{2}
& \eqref{F.3}: \quad & -2cf + bf + 2(cd - be) = -6f + 2f + 6d
< -4f + 4f &= 0; \\ 
& \eqref{F.4}: \quad & 2cf + f(1-b) + 2(e(b-1) + c(1-d)) &= \\
& & 6f - f + 6(1-d) \leq 3f + 2f + 6 - 6\frac{4}{9}f = 3f +
\frac{2}{3}(9-f) &< 3f.
\end{alignat*}
Thus, let $d \geq \frac{2}{3}f$. The point $(3,1,3)$ is now in the
interior of $S$: clearly, \eqref{F.1} and \eqref{F.2} are strict;
\begin{alignat*}{2}
& \eqref{F.3}: \quad & -3cf + bf + 3(cd - be) = -9f + 2f + 9d
\leq -7f + 9(\frac{2}{3}f + 1) = 9-f &= 0; \\ 
& \eqref{F.4}: \quad & 3cf + f(1-b) + 3(e(b-1) + c(1-d)) &= \\
& & 9f - f + 9(1-d) \leq 3f + 5f + 9 - 6f = 3f + 9 - f &< 3f.
\end{alignat*}

$\bullet$ Let $(b,c) = (3,4)$. We show that $f \leq 8$. So assume $f \geq
9$. From \eqref{point.1} and \eqref{point.2} we obtain $f \leq 4d$
and $f \geq 2(d-1)$. If $2d < f$ holds true, then the point
$(2,2,1)$ is in the interior of $S$: clearly, \eqref{F.1} and
\eqref{F.2} are strict;
\begin{alignat*}{2}
& \eqref{F.3}: \quad & -2cf + 2bf + cd - be = -8f + 6f + 4d
= 2(2d-f) &< 0; \\ 
& \eqref{F.4}: \quad & 2cf + 2f(1-b) + e(b-1) + c(1-d) &= \\
& & 8f - 4f + 4(1-d) < 4f + 4 - f &< 4f.
\end{alignat*}
Thus, let $2d \geq f$. The point $(2,1,2)$ is now in the
interior of $S$: clearly, \eqref{F.1} and \eqref{F.2} are strict;
\begin{alignat*}{2}
& \eqref{F.3}: \quad & -2cf + bf + 2(cd - be) = -8f + 3f + 8d
\leq -5f + 8(\frac{1}{2}f + 1) = 8-f &< 0; \\ 
& \eqref{F.4}: \quad & 2cf + f(1-b) + 2(e(b-1) + c(1-d)) &= \\
& & 8f - 2f + 8(1-d) \leq 4f + 2f + 8 - 4f = 4f + 2(4 - f) &< 4f.
\end{alignat*}

$\bullet$ Let $(b,c) = (4,5)$. We show that $f \leq 5$. So assume $f \geq
6$. From \eqref{point.1} and \eqref{point.2} we obtain $f \leq 5d$
and $2f \geq 5(d-1)$. If $5d < 2f$ holds true, then the point
$(2,2,1)$ is in the interior of $S$: clearly, \eqref{F.1} and
\eqref{F.2} are strict;
\begin{alignat*}{2}
& \eqref{F.3}: \quad & -2cf + 2bf + cd - be = -10f + 8f + 5d
= -2f + 5d &< 0; \\ 
& \eqref{F.4}: \quad & 2cf + 2f(1-b) + e(b-1) + c(1-d) &= \\
& & 10f - 6f + 5(1-d) = 5f - f + 5(1-d) &< 5f.
\end{alignat*}
Thus, let $5d \geq 2f$. The point $(2,1,2)$ is now in the
interior of $S$: clearly, \eqref{F.1} and \eqref{F.2} are strict;
\begin{alignat*}{2}
& \eqref{F.3}: \quad & -2cf + bf + 2(cd - be) = -10f + 4f + 10d
\leq -6f + 10(\frac{2}{5}f + 1) = 2(5-f) &< 0; \\ 
& \eqref{F.4}: \quad & 2cf + f(1-b) + 2(e(b-1) + c(1-d)) &= \\
& & 10f - 3f + 10(1-d) \leq 5f + 2f + 10 - 4f = 5f + 2(5 - f) &< 5f.
\end{alignat*}

$\bullet$ Let $(b,c) = (5,6)$. We show that $f \leq 6$. So assume $f \geq
7$. From \eqref{point.1} and \eqref{point.2} we obtain $f \leq 6d$
and $f \geq 3(d-1)$. If $3d < f$ holds true, then the point
$(2,2,1)$ is in the interior of $S$: clearly, \eqref{F.1} and
\eqref{F.2} are strict;
\begin{alignat*}{2}
& \eqref{F.3}: \quad & -2cf + 2bf + cd - be = -12f + 10f + 6d
= 2(3d-f) &< 0; \\ 
& \eqref{F.4}: \quad & 2cf + 2f(1-b) + e(b-1) + c(1-d) &= \\
& & 12f - 8f + 6(1-d) = 6f - 2f + 6 - f &< 6f.
\end{alignat*}
Thus, let $3d \geq f$. The point $(2,1,2)$ is now in the
interior of $S$: clearly, \eqref{F.1} and \eqref{F.2} are strict;
\begin{alignat*}{2}
& \eqref{F.3}: \quad & -2cf + bf + 2(cd - be) = -12f + 5f + 12d
\leq -7f + 12(\frac{1}{3}f + 1) = 3(4-f) &< 0; \\ 
& \eqref{F.4}: \quad & 2cf + f(1-b) + 2(e(b-1) + c(1-d)) &= \\
& & 12f - 4f + 12(1-d) \leq 6f + 2f + 12 - 4f = 6f + 2(6 - f) &< 6f.
\end{alignat*}

$\bullet$ Let $(b,c) = (2,4)$. We show that $f \leq 12$. So assume $f \geq
13$. From \eqref{point.1} and \eqref{point.2} we obtain $f \leq 2d$
and $3f \geq 4(d-1)$. If $4d < 3f$ holds true, then the point
$(2,1,2)$ is in the interior of $S$: clearly, \eqref{F.1} and
\eqref{F.2} are strict;
\begin{alignat*}{2}
& \eqref{F.3}: \quad & -2cf + bf + 2(cd - be) = -8f + 2f + 8d
= 2(4d-3f) &< 0; \\ 
& \eqref{F.4}: \quad & 2cf + f(1-b) + 2(e(b-1) + c(1-d)) &= \\
& & 8f - f + 8(1-d) \leq 4f + 3f + 8 - 4f = 4f + 8 - f &< 4f.
\end{alignat*}
Thus, let $4d \geq 3f$. The point $(3,1,3)$ is now in the
interior of $S$: clearly, \eqref{F.1} and \eqref{F.2} are strict;
\begin{alignat*}{2}
& \eqref{F.3}: \quad & -3cf + bf + 3(cd - be) = -12f + 2f + 12d
\leq -10f + 12(\frac{3}{4}f + 1) = 12 - f &< 0; \\ 
& \eqref{F.4}: \quad & 3cf + f(1-b) + 3(e(b-1) + c(1-d)) &= \\
& & 12f - f + 12(1-d) \leq 4f + 7f + 12 - 9f = 4f + 2(6 - f) &< 4f.
\end{alignat*}

$\bullet$ Let $(b,c) = (3,5)$. We show that $f \leq 10$. So assume $f \geq
11$. From \eqref{point.1} and \eqref{point.2} we obtain $2f \leq 5d$
and $3f \geq 5(d-1)$. The point $(2,1,2)$ is now in the
interior of $S$: clearly, \eqref{F.1} and \eqref{F.2} are strict;
\begin{alignat*}{2}
& \eqref{F.3}: \quad & -2cf + bf + 2(cd - be) = -10f + 3f + 10d
\leq -7f + 10(\frac{3}{5}f + 1) = 10 - f &< 0; \\ 
& \eqref{F.4}: \quad & 2cf + f(1-b) + 2(e(b-1) + c(1-d)) &= \\
& & 10f - 2f + 10(1-d) \leq 5f + 3f + 10 - 4f = 5f + 10 - f &< 5f.
\end{alignat*}

$\bullet$ Let $(b,c) = (4,6)$. We show that $f \leq 12$. So assume $f \geq
13$. From \eqref{point.1} and \eqref{point.2} we obtain $f \leq 3d$
and $f \geq 2(d-1)$. The point $(2,1,2)$ is now in the
interior of $S$: clearly, \eqref{F.1} and \eqref{F.2} are strict;
\begin{alignat*}{2}
& \eqref{F.3}: \quad & -2cf + bf + 2(cd - be) = -12f + 4f + 12d
\leq -8f + 12(\frac{1}{2}f + 1) = 2(6-f) &< 0; \\ 
& \eqref{F.4}: \quad & 2cf + f(1-b) + 2(e(b-1) + c(1-d)) &= \\
& & 12f - 3f + 12(1-d) \leq 6f + 3f + 12 - 4f = 6f + 12 - f &< 6f.
\end{alignat*}

$\bullet$ Let $(b,c) = (2,5)$. We show that $f \leq 5$. So assume $f \geq
6$. From \eqref{point.1} and \eqref{point.2} we obtain $3f \leq 5d$
and $4f \geq 5(d-1)$. If $5d < 4f$ holds true, then the point
$(2,1,2)$ is in the interior of $S$: clearly, \eqref{F.1} and
\eqref{F.2} are strict;
\begin{alignat*}{2}
& \eqref{F.3}: \quad & -2cf + bf + 2(cd - be) = -10f + 2f + 10d
= 2(5d-4f) &< 0; \\ 
& \eqref{F.4}: \quad & 2cf + f(1-b) + 2(e(b-1) + c(1-d)) &= \\
& & 10f - f + 10(1-d) \leq 5f + 4f + 10 - 6f = 5f + 2(5-f) &< 5f.
\end{alignat*}
Thus, let $5d \geq 4f$. The point $(2,2,1)$ is now in the
interior of $S$: clearly, \eqref{F.1} and \eqref{F.2} are strict;
\begin{alignat*}{2}
& \eqref{F.3}: \quad & -2cf + 2bf + cd - be = -10f + 4f + 5d
\leq -6f + 5(\frac{4}{5}f + 1) = 5 - 2f &< 0; \\ 
& \eqref{F.4}: \quad & 2cf + 2f(1-b) + e(b-1) + c(1-d) &= \\
& & 10f - 2f + 5(1-d) \leq 5f + 3f + 5 - 4f = 5f + 5 - f &< 5f.
\end{alignat*}

$\bullet$ Let $(b,c) = (2,6)$. We show that $(2,1,2)$ or $(2,3,1)$ is in the
interior of $S$. Note that $f \geq c = 6$, by assumption. From
\eqref{point.1} and \eqref{point.2} we obtain $2f \leq 3d$ 
and $5f \geq 6(d-1)$. If $6d < 5f$ holds true, then the point
$(2,1,2)$ is in the interior of $S$: clearly, \eqref{F.1} and
\eqref{F.2} are strict;
\begin{alignat*}{2}
& \eqref{F.3}: \quad & -2cf + bf + 2(cd - be) = -12f + 2f + 12d
= 2(6d-5f) &< 0; \\ 
& \eqref{F.4}: \quad & 2cf + f(1-b) + 2(e(b-1) + c(1-d)) &= \\
& & 12f - f + 12(1-d) \leq 6f + 5f + 12 - 8f = 6f + 3(4-f) &< 6f.
\end{alignat*}
Thus, let $6d \geq 5f$. The point $(2,3,1)$ is now in the
interior of $S$: clearly, \eqref{F.1} and \eqref{F.2} are strict;
\begin{alignat*}{2}
& \eqref{F.3}: \quad & -2cf + 3bf + cd - be = -12f + 6f + 6d
= 6(d-f) &< 0; \\ 
& \eqref{F.4}: \quad & 2cf + 3f(1-b) + e(b-1) + c(1-d) &= \\
& & 12f - 3f + 6(1-d) \leq 6f + 3f + 6 - 5f = 6f + 2(3-f) &< 6f.
\end{alignat*}

$\bullet$ Let $(b,c) = (2,7)$. We show that $(2,1,2)$ or $(2,3,1)$ is in the
interior of $S$. Note that $f \geq c = 7$, by assumption. From
\eqref{point.1} and \eqref{point.2} we obtain $5f \leq 7d$ 
and $6f \geq 7(d-1)$. If $7d < 6f$ holds true, then the point
$(2,1,2)$ is in the interior of $S$: clearly, \eqref{F.1} and
\eqref{F.2} are strict;
\begin{alignat*}{2}
& \eqref{F.3}: \quad & -2cf + bf + 2(cd - be) = -14f + 2f + 14d
= 2(7d - 6f) &< 0; \\ 
& \eqref{F.4}: \quad & 2cf + f(1-b) + 2(e(b-1) + c(1-d)) &= \\
& & 14f - f + 14(1-d) \leq 7f + 6f + 14 - 10f = 7f + 14 - 4f &< 7f.
\end{alignat*}
Thus, let $7d \geq 6f$. The point $(2,3,1)$ is now in the
interior of $S$: clearly, \eqref{F.1} and \eqref{F.2} are strict;
\begin{alignat*}{2}
& \eqref{F.3}: \quad & -2cf + 3bf + cd - be = -14f + 6f + 7d
= 7(d-f) - f &< 0; \\ 
& \eqref{F.4}: \quad & 2cf + 3f(1-b) + e(b-1) + c(1-d) &= \\
& & 14f - 3f + 7(1-d) \leq 7f + 4f + 7 - 6f = 7f + 7 - 2f &< 7f.
\end{alignat*}

$\bullet$ Let $(b,c) = (3,6)$. We show that $f \leq 12$. So assume $f \geq
13$. From \eqref{point.1} and \eqref{point.2} we obtain $f \leq 2d$
and $2f \geq 3(d-1)$. The point $(2,1,2)$ is now in the
interior of $S$: clearly, \eqref{F.1} and \eqref{F.2} are strict;
\begin{alignat*}{2}
& \eqref{F.3}: \quad & -2cf + bf + 2(cd - be) = -12f + 3f + 12d
\leq -9f + 12(\frac{2}{3}f + 1) = 12 - f &< 0; \\ 
& \eqref{F.4}: \quad & 2cf + f(1-b) + 2(e(b-1) + c(1-d)) &= \\
& & 12f - 2f + 12(1-d) \leq 6f + 4f + 12 - 6f = 6f + 2(6 - f) &< 6f.
\end{alignat*}

$\bullet$ Let $(b,c) = (3,7)$. We show that $f \leq 14$. So assume $f \geq
15$. From \eqref{point.1} and \eqref{point.2} we obtain $4f \leq 7d$
and $5f \geq 7(d-1)$. The point $(2,1,2)$ is now in the
interior of $S$: clearly, \eqref{F.1} and \eqref{F.2} are strict;
\begin{alignat*}{2}
& \eqref{F.3}: \quad & -2cf + bf + 2(cd - be) = -14f + 3f + 14d
\leq -11f + 14(\frac{5}{7}f + 1) = 14 - f &< 0; \\ 
& \eqref{F.4}: \quad & 2cf + f(1-b) + 2(e(b-1) + c(1-d)) &= \\
& & 14f - 2f + 14(1-d) \leq 7f + 5f + 14 - 8f = 7f + 14 - 3f &< 7f.
\end{alignat*}

$\bullet$ Let $(b,c) = (3,8)$. We show that $f \leq 16$. So assume $f \geq
17$. From \eqref{point.1} and \eqref{point.2} we obtain $5f \leq 8d$
and $3f \geq 4(d-1)$. The point $(2,1,2)$ is now in the
interior of $S$: clearly, \eqref{F.1} and \eqref{F.2} are strict;
\begin{alignat*}{2}
& \eqref{F.3}: \quad & -2cf + bf + 2(cd - be) = -16f + 3f + 16d
\leq -13f + 16(\frac{3}{4}f + 1) = 16 - f &< 0; \\ 
& \eqref{F.4}: \quad & 2cf + f(1-b) + 2(e(b-1) + c(1-d)) &= \\
& & 16f - 2f + 16(1-d) \leq 8f + 6f + 16 - 10f = 8f + 4(4 - f) &< 8f.
\end{alignat*}

\subsubsection*{2) $\boldsymbol{e > 0}$}

\paragraph*{i) $\boldsymbol{c \leq e}$}

We first show that in this case we must have $c = e$. For the purpose
of deriving a contradiction assume that $c < e$. Then, the point
$(2,2,1)$ is in the interior of $S$: clearly, \eqref{F.1} and
\eqref{F.2} are strict;
\begin{alignat*}{2}
& \eqref{F.3}: \quad & -2cf + 2bf + cd - be \leq 2f(b-c) + f(c-b) + f
+ c - e = f (b + 1 - c) + c - e &< 0; \\ 
& \eqref{F.4}: \quad & 2cf + 2f(1-b) + e(b-1) + c(1-d) &\leq \\
& & 2f(c + 1 - b) + c - e + f(b-c) = cf + f(2-b) + c - e &< cf.
\end{alignat*}
The inequalities in the first row follow from \eqref{(2,1,1)} and the
fact that $b < c < e$, whereas the inequality in the second row
follows from \eqref{(1,1,1)} and the inequality in the last row
follows from $2 \leq b$ and $c < e$. Therefore, we have $c = e$.

Consider the sequence of points $(2,\theta,1)$, where $\theta \geq 1$.
By equation \eqref{(2,1,1)} and the fact that $b \geq 2$, we obtain
that there exists some $\Theta \geq 2$ such that
\begin{alignat}{2}
 2cf + f(1-b)(\Theta - 1) + e(b-1) + c(1-d) \geq cf&, \label{sequ.2-1} \\
 2cf + f(1-b)\Theta + e(b-1) + c(1-d) < cf&. 
\end{alignat}
Since the point $(2,\Theta,1)$ satisfies \eqref{F.1},
\eqref{F.2}, and  \eqref{F.4} strictly, we must have that
\begin{equation} \label{sequ.2-3}
  -2cf + bf\Theta + cd - be \geq 0
\end{equation}
since otherwise the point $(2,\Theta,1)$ is in the interior of
$S$. Adding \eqref{sequ.2-1} and \eqref{sequ.2-3} yields  $0 \leq
f(\Theta + b - 1 - c) + c - e = f(\Theta + b - 1 - c)$ which implies 
\begin{equation} \label{sequ.2-4}
  \Theta \geq c - b + 1.
\end{equation}
Using \eqref{(1,1,1)} and \eqref{sequ.2-1}, we obtain $f(c-b) \leq cd
- be \leq (\Theta - 1)f(1-b) + cf + c - e = (\Theta - 1)f(1-b) +
cf$ and hence $f(c-b) \leq (\Theta - 1)f(1-b) + cf \Leftrightarrow
\Theta \leq 1 + \frac{b}{b-1}$. We infer
\begin{equation} \label{bound.on.c}
  \Theta \leq
  \begin{cases}
    3, & \textup{if } b = 2 \\
    2, & \textup{if } b \geq 3,
  \end{cases} \quad \stackrel{\eqref{sequ.2-4}}{\Longrightarrow} \quad
  c \leq \Theta + b - 1  \leq
  \begin{cases}
    4, & \textup{if } b = 2 \\
    b + 1, & \textup{if } b \geq 3.
  \end{cases}
\end{equation}
This shows that for $b = 2$ only the two cases $(b,c) = (2,3)$ and
$(b,c) = (2,4)$ need to be considered. Since $b < c$ the case $b \geq
3$ leads to $b + 1 \leq c \leq b + 1$ and thus only the case $(b,c) =
(c-1,c)$ is left.
\bigskip

$\bullet$ Let $(b,c) = (2,3)$. We show that $f \leq 9$. So assume $f \geq
10$. From \eqref{(1,1,1)} and \eqref{(2,1,1)} we obtain $f \leq 3(d-2)$
and $2f \geq 3(d-2)$. If $-4f + 9(d-2) < 0$ holds true, then the point
$(2,1,3)$ is in the interior of $S$: clearly, \eqref{F.1} is strict;
\eqref{F.2} is strict since $e = c = 3$ and $f \geq 10$;
\begin{alignat*}{2}
& \eqref{F.3}: \quad & -2cf + bf + 3(cd - be) = -6f + 2f + 3(3d - 6)
= -4f + 9(d-2) &< 0; \\ 
& \eqref{F.4}: \quad & 2cf + f(1-b) + 3(e(b-1) + c(1-d)) = 6f - f +
3(6-3d) &= \\
& & 3f + 2f + 18 - 9d \leq 3f + 2f + 18 -
9(\frac{1}{3}f + 2) = 3f - f &< 3f.
\end{alignat*}
Thus, let $-4f + 9(d-2) \geq 0$. If $-4f + 6(d-2) < 0$ holds true,
then the point $(2,1,2)$ is in the interior of $S$: clearly,
\eqref{F.1} and \eqref{F.2} are strict;
\begin{alignat*}{2}
& \eqref{F.3}: \quad & -2cf + bf + 2(cd - be) = -6f + 2f + 2(3d-6)
= -4f + 6(d-2) &< 0; \\ 
& \eqref{F.4}: \quad & 2cf + f(1-b) + 2(e(b-1) + c(1-d)) = 6f - f +
2(6-3d) &= \\
& & 3f + 2f + 12 - 6d \leq 3f + 2f + 12 -
6(\frac{4}{9}f + 2) = 3f - \frac{2}{3}f &< 3f.
\end{alignat*}
Thus, let $-4f + 6(d-2) \geq 0 \Leftrightarrow 3(d-2) \geq 2f$. Since
also $3(d-2) \leq 2f$ holds true we obtain $3(d-2) = 2f$. The point
$(3,1,3)$ is now in the interior of $S$: clearly, \eqref{F.1} and
\eqref{F.2} are strict;
\begin{alignat*}{2}
& \eqref{F.3}: \quad & -3cf + bf + 3(cd - be) = -9f + 2f + 3(3d-6)
= -7f + 6f = -f &< 0; \\ 
& \eqref{F.4}: \quad & 3cf + f(1-b) + 3(e(b-1) + c(1-d)) =
9f - f + 3(6-3d) = 2f &< 3f.
\end{alignat*}

$\bullet$ Let $(b,c) = (2,4)$. We show that $f \leq 8$. So assume $f \geq
9$. From \eqref{(1,1,1)} and \eqref{(2,1,1)} we obtain $f + 4 \leq 2d$
and $3f + 8 \geq 4d$. If $4d < 3f + 8$ holds true, then the point
$(2,1,2)$ is in the interior of $S$: clearly, \eqref{F.1} is strict;
\eqref{F.2} is strict since $e = c = 4$ and $f \geq 9$;
\begin{alignat*}{2}
& \eqref{F.3}: \quad & -2cf + bf + 2(cd - be) = -8f + 2f + 2(4d - 8)
= 2(4d - 3f - 8) &< 0; \\ 
& \eqref{F.4}: \quad & 2cf + f(1-b) + 2(e(b-1) + c(1-d)) &= \\
& & 8f - f + 2(8-4d) = 4f + 3f + 16 - 8(\frac{1}{2}f + 2) = 4f - f &<
4f.
\end{alignat*}
Thus, using \eqref{(2,1,1)} we have $4d \geq 3f + 8 \geq 4d$ which
implies $4d = 3f + 8$. The point $(2,2,1)$ is now in the interior of
$S$: clearly, \eqref{F.1} and \eqref{F.2} are strict;
\begin{alignat*}{2}
& \eqref{F.3}: \quad & -2cf + 2bf + cd - be = -8f + 4f + 4d - 8
= -f &< 0; \\ 
& \eqref{F.4}: \quad & 2cf + 2f(1-b) + e(b-1) + c(1-d) =
8f - 2f + 8 - 4d = 3f &< 4f.
\end{alignat*}

$\bullet$ Let $(b,c) = (c-1,c)$ with $b \geq 3$. Since $c - e \leq 0 <
f$ in this case, it follows from \eqref{bound.on.b} that $b \leq 5$
and therefore $4 \leq c = b + 1 \leq 6$. We show that $f \leq 12$. So
assume $f \geq 13$. From \eqref{(1,1,1)} and \eqref{(2,1,1)} we obtain
$f \leq c(d + 1 - c)$ and $2f \geq c(d + 1 - c)$. If $f < c(d + 1 - c)$
holds true, then the point $(2,1,2)$ is in the interior of $S$:
clearly, \eqref{F.1} is strict; \eqref{F.2} is strict since $c \leq 6$
and $f \geq 13$;
\begin{alignat*}{2}
& \eqref{F.3}: \quad & -2cf + bf + 2(cd - be) &= \\
& & -2cf + (c-1)f + 2c(d + 1 - c) \leq -f(c + 1) + 4f = f(3-c) &< 0; \\
& \eqref{F.4}: \quad & 2cf + f(1-b) + 2(e(b-1) + c(1-d)) &= \\
& & 2cf + (2-c)f + 2c(c - d - 1) < (2 + c)f - 2f &= cf.
\end{alignat*}
Thus, using \eqref{(1,1,1)} we have $f \geq c(d + 1 - c) \geq f$ which
implies $f = c(d + 1 - c)$. The point $(2,2,1)$ is now in the interior of
$S$: clearly, \eqref{F.1} and \eqref{F.2} are strict;
\begin{alignat*}{2}
& \eqref{F.3}: \quad & -2cf + 2bf + cd - be =
-2cf + 2(c-1)f + c(d + 1 - c) = -f &< 0; \\
& \eqref{F.4}: \quad & 2cf + 2f(1-b) + e(b-1) + c(1-d) &= \\
& & 2cf + 2(2-c)f + c(c - d - 1) = 3f &< cf.
\end{alignat*}

\paragraph*{ii) $\boldsymbol{c > e}$}

We show that by using a unimodular transformation this case can be
reduced to a case which has already been analyzed. Assume that the
vertices of $S$ are given by the columns of the matrix
\begin{equation} \label{pre.unimod}
  \begin{pmatrix}
    0 & 1 & b & d \\
    0 & 0 & c & e \\
    0 & 0 & 0 & f
  \end{pmatrix},
\end{equation}
where - besides the usual conditions on the unknowns $b, c, d, e$, and
$f$ - in addition $c$ is chosen such that it is minimal with respect
to all such representations. In the following we show that
\eqref{pre.unimod} is unimodularly transformable to a simplex where
the parameter $a$ (here: $a = 1$) is greater than or equal to $2$.
\bigskip

Let $r := \gcd(e,f)$ and $g := \gcd(r,d)$ where we assume $r,g >
0$. Then, there exist $p,q,s,t \in \Z$ such that $r = pe + qf$ and $g
= sr + td$. Apply the following unimodular transformation to
\eqref{pre.unimod}:
\begin{equation*}
  \begin{pmatrix} x_1 \\ x_2 \\ x_3 \end{pmatrix} \mapsto
  \begin{pmatrix}
    -t & -sp & -sq \\
    \frac{r}{g} & -\frac{dp}{g} & -\frac{dq}{g} \\
    0 & \frac{f}{r} & -\frac{e}{r}
  \end{pmatrix}
  \begin{pmatrix} x_1 \\ x_2 \\ x_3 \end{pmatrix} +
  \begin{pmatrix} g \\ 0 \\ 0 \end{pmatrix}.
\end{equation*}
The columns of the matrix listed below represent the vertices of the
resulting simplex:
\begin{equation} \label{post.unimod}
  \begin{pmatrix}
    0 & g & g-t & g - tb - psc \\
    0 & 0 & \frac{r}{g} & \frac{rb - pcd}{g} \\
    0 & 0 & 0 & \frac{cf}{r}
  \end{pmatrix}.
\end{equation}
If $g = 1$, then \eqref{post.unimod} can be transformed by elementary
row operations into Hermit normal form. However, this does not change
the diagonal elementes and thus leads to a representation
\eqref{pre.unimod} where $\frac{r}{g} = r = \gcd(e,f) \leq e <
c$. This contradicts the minimal choice of $c$. Hence, let $g \geq 2$.
Using elementary row operations \eqref{post.unimod} can be brought
into Hermit normal form with $a = g \geq 2$. Such simplices were
analyzed in Section \ref{case.dist.>=2}.

\subsection*{Acknowledgements}
We are grateful to G.~Averkov, M.~Henk, and S.~Onn for interesting
discussions on the topic of the paper and pointers to literature.


\small
\bibliography{vers2}

\begin{thebibliography}{10}

\bibitem{anlowe2}
K.~Andersen, Q.~Louveaux, and R.~Weismantel.
\newblock An analysis of mixed integer linear sets based on lattice point free
  convex sets.
\newblock {\em Submitted to Mathematics of Operations Research}, 2008.

\bibitem{anlowewo}
K.~Andersen, Q.~Louveaux, R.~Weismantel, and L.~Wolsey.
\newblock Inequalities from two rows of a simplex tableau.
\newblock {\em IPCO conference 2007, Lecture Notes in Computer Science 4513,
  Springer}, pages 1--15, 2007.

\bibitem{balasint}
E.~Balas.
\newblock Intersection cuts - a new type of cutting planes for integer
  programming.
\newblock {\em Operations Research}, 19:19--39, 1971.

\bibitem{coma}
G.~Cornu\'ejols and F.~Margot.
\newblock On the facets of mixed integer programs with two integer variables
  and two constraints.
\newblock {\em LATIN conference 2008, Lecture Notes in Computer Science 4957,
  Springer}, pages 317--328, 2008.

\bibitem{dewo}
S.S. Dey and L.~Wolsey.
\newblock Lifting integer variables in minimal inequalities corresponding to
  lattice-free triangles.
\newblock {\em IPCO conference 2008, Lecture Notes in Computer Science 5035,
  Springer}, pages 463--475, 2008.

\bibitem{lazi}
J.C. Lagarias and G.M. Ziegler.
\newblock Bounds for lattice polytopes containing a fixed number of interior
  points in a sublattice.
\newblock {\em Canadian Journal of Mathematics}, 43:1022--1035, 1991.

\bibitem{lovasz}
L.~Lov\'asz.
\newblock Geometry of numbers and integer programming.
\newblock {\em Mathematical Programming, Recent Developments and Applications},
  pages 177--201, 1989.

\bibitem{reznick}
B.~Reznick.
\newblock Lattice point simplices.
\newblock {\em Discrete Mathematics}, 60:219--242, 1986.

\bibitem{scarf}
H.E. Scarf.
\newblock Integral polyhedra in three space.
\newblock {\em Mathematics of Operations Research}, 10:403--438, 1985.

\bibitem{Schrijver}
A.~Schrijver.
\newblock Theory of linear and integer programming.
\newblock {\em Wiley, Chichester}, 1986.

\bibitem{sebo}
A.~Seb{\"o}.
\newblock An introduction to empty lattice simplices.
\newblock {\em IPCO conference 1999, Lecture Notes in Computer Science 1610,
  Springer}, pages 400--414, 1999.

\bibitem{treutlein}
J.~Treutlein.
\newblock 3-dimensional lattice polytopes without interior lattice points.
\newblock {\em arXiv:0809.1787}, 2008.

\bibitem{zambelli}
G.~Zambelli.
\newblock On degenerate multi-row \mbox{Gomory} cuts.
\newblock {\em Operations Research Letters}, 37:21--22, 2009.

\end{thebibliography}
\bibliographystyle{plain}

\end{document}